\documentclass[11pt,final]{amsart}

\usepackage{amsmath, amsxtra, amssymb, mathdots}
\usepackage{appendix}
\usepackage{verbatim}
\usepackage{multirow}

\usepackage{graphicx}

\usepackage[cmtip,all]{xy}

\usepackage[notcite, notref]{showkeys}

\usepackage[colorlinks, linkcolor=blue, citecolor=red, urlcolor=black]{hyperref}
\usepackage[margin=1.5in]{geometry}
\newtheorem{theorem}{Theorem}[section]
\newtheorem{lemma}[theorem]{Lemma}
\newtheorem{prop}[theorem]{Proposition}
\newtheorem{corollary}[theorem]{Corollary}
\newtheorem{claim}[theorem]{Claim}

\theoremstyle{definition}

\theoremstyle{remark}
\newtheorem{remark}[theorem]{Remark}

\numberwithin{equation}{section}

\DeclareMathOperator{\spec}{Spec}

\DeclareMathOperator{\Sym}{Sym}

\DeclareMathOperator{\Stab}{Stab}

\DeclareMathOperator{\init}{in}

\DeclareMathOperator{\SL}{SL}
\DeclareMathOperator{\GL}{GL}
\DeclareMathOperator{\PGL}{PGL}
\DeclareMathOperator{\codim}{codim}

\DeclareMathOperator{\Grass}{Grass}
\DeclareMathOperator{\Gr}{Gr}

\DeclareMathOperator{\frakRes}{\mathfrak{Res}}
\DeclareMathOperator{\DS}{\mathfrak{DS}}

\DeclareMathOperator{\Res}{Res}

\def\VV{\mathbb{V}}
\def\AA{\mathbf{A}}

\newcommand{\gitq}{/\hspace{-0.25pc}/}

\def\A{\mathcal{A}}

\def\twohead{\twoheadrightarrow}

\def\cA{\mathcal{A}}

\def\PP{\mathbb{P}}

\def\ZZ{\mathbb{Z}}
\def\CC{\mathbb{C}}
\def\HH{\mathrm{H}}
\def\cO{\mathcal{O}}
\def\cU{\mathcal{U}}

\def\cN{\mathcal N}

\def\D{\mathcal{D}}

\DeclareMathOperator{\Ann}{Ann}

\DeclareMathOperator{\Image}{Im}

\makeatletter
\def\blfootnote{\xdef\@thefnmark{}\@footnotetext}
\makeatother

\begin{document}

\title[Associated form morphism]{Associated form morphism}

\author[Maksym Fedorchuk]{Maksym Fedorchuk}
\address[Fedorchuk]{Department of Mathematics\\
Boston College\\
140 Commonwealth Ave\\
Chestnut Hill, MA 02467, USA}
\email{maksym.fedorchuk@bc.edu}

\author[Alexander Isaev]{Alexander Isaev}
\address[Isaev]{Mathematical Sciences Institute\\
Australian National University\\
Acton, Canberra, ACT 2601, Australia}
\email{alexander.isaev@anu.edu.au}

\thanks{The first author was partially supported by the NSA Young Investigator grant H98230-16-1-0061 
and Alfred P. Sloan Research Fellowship.}

\begin{abstract} We study the geometry of the morphism that sends 
a smooth hypersurface of degree $d+1$ in $\PP^{n-1}$ to its associated hypersurface of degree $n(d-1)$
in the dual space $\bigl(\PP^{n-1}\bigr)^\vee$.
\end{abstract}

\maketitle

\setcounter{tocdepth}{1}
\tableofcontents

\section{Introduction} 

One of the first applications of Geometric Invariant Theory is a construction
of the moduli space of smooth degree $m$ hypersurfaces in a fixed projective space 
$\PP^{n-1}$ \cite{GIT}. This moduli space is an affine GIT quotient 
\[
U_{m,n}:=\left(\PP\HH^0\bigl(\PP^{n-1}, \cO(m)\bigr) \setminus {\Delta}\right) \gitq \PGL(n),\label{E:smooth-GIT}
\]
where $\Delta$ is the discriminant divisor parameterizing singular hypersurfaces.
The GIT construction produces a natural compactification
\[
U_{m,n} \subset V_{m,n}:=\left(\PP\HH^0\bigl(\PP^{n-1}, \cO(m)\bigr)\right)^{ss} \gitq \PGL(n),\label{E:compact-GIT}
\]
given by a categorical quotient of the locus of GIT semistable hypersurfaces. We call $V_{m,n}$ the 
GIT compactification of $U_{m,n}$.

The subject of this paper is a certain rational map $V_{m,n} \dashrightarrow V_{n(m-2),n}$, where $n\ge 2$, $m\ge 3$ and where we exclude the (trivial) case $(n,m)=(2,3)$. While this
map has a purely algebraic construction, which we shall recall soon, it has several surprising geometric properties that we
establish in this paper. In particular, this rational map restricts 
to a locally closed immersion $\bar A\colon U_{m,n}\to V_{n(m-2),n}$, 
and often contracts the discriminant divisor in $V_{m,n}$. Consequently, the closure of the image of $\bar A$ in $V_{n(m-2),n}$
is a compactification of the GIT moduli space $U_{m,n}$ that is different from the GIT compactification $V_{m,n}$. 

To define $\bar A$, we consider the \emph{associated form morphism} defined on the space
of smooth homogeneous forms $f\in \CC[x_1,\dots,x_n]$ of fixed degree $m\ge 3$.
Given such an $f$,  its associated form $A(f)$ is a degree $n(m-2)$ homogeneous form
in the graded dual polynomial ring $\CC[z_1,\dots, z_n]$. In our recent paper \cite{fedorchuk-isaev}, we proved that 
the associated form $A(f)$ is always polystable in the sense of GIT. Consequently, we obtain a
morphism $\bar A$ from $U_{m,n}$ to $V_{n(m-2),n}$ sending the image of $f$ in $U_{m,n}$ to the image of $A(f)$ in $V_{n(m-2),n}$.

Our first result is that the morphism $\bar A$ is an isomorphism onto its image, a locally closed subvariety in the target. 
\begin{theorem}
\label{T:barA}
The morphism 
\[
\bar A \colon U_{m,n} \to V_{n(m-2),n} 
\]
is a locally closed immersion.
\end{theorem}
In the process of establishing Theorem \ref{T:barA}, we generalize results of \cite{alper-isaev-assoc-binary} to the case of an arbitrary number of variables,
and, in particular, prove that the auxiliary gradient morphism sending a semistable form to the span of its partial derivatives 
gives rise to a closed immersion on the level of quotients (see Theorem \ref{T:gradient}). 

Our second main result is Theorem \ref{T:barA1}, which describes the rational map\linebreak $\bar A \colon V_{m,n} \dashrightarrow V_{n(m-2),n}$ in codimension one. Namely, we study how $\bar A$ extends to the generic point of the discriminant divisor in the GIT
compactification (see Corollary \ref{cor:n-1,n}), and prove that for $n=2,3$ and $m\ge 4$, as well as for $n\geq 4$, $m\gg 0$, the morphism $\bar A$ contracts the discriminant divisor to a lower-dimensional subvariety in the target (see Corollary \ref{C:contraction}). In the process, we prove that the image of $\bar A$ contains the orbit of the Fermat hypersurface in its closure and as a result obtain a new proof of the generic smoothness of associated forms (see Corollary \ref{C:generic-smoothness}).


\subsection{Notation and conventions}
Let $S:=\Sym V\simeq \CC[x_1,\dots, x_n]$ be a symmetric algebra of
an $n$-dimensional vector space $V$, with its standard grading. Let $\D:=\Sym V^{\vee} \simeq \CC[z_1,\dots, z_n]$
be the graded dual of $S$, with the structure of the $S$-module given by
the \emph{polar pairing} $S\times \D \to \D$, which is defined by
\begin{equation}\label{E:apolarity-action}
g(x_1,\dots,x_n)\circ F(z_1,\dots,z_n):=g(\partial/\partial z_1, \dots, \partial/\partial z_1)F(z_1,\dots,z_n).
\end{equation}

A homogeneous polynomial $f\in S_{m}$ is called a \emph{direct sum} if, after a linear change of variables, it can be 
written as the sum of two non-zero polynomials in 
disjoint sets of variables: 
\[
f=f_1(x_1,\dots, x_a)+f_2(x_{a+1},\dots, x_n).\label{E:direct-sum-def}
\]
We will use the recognition criteria for direct sums established in \cite{fedorchuk-direct}, and so we keep the pertinent terminology 
of that paper. We will say that $f\in S_{m}$ is a $k$-partial Fermat form for some $k\leq n$, if, after a linear change of variables, 
it can be written as follows: 
\[
f=x_1^{m}+\cdots+x_k^{m}+g(x_{k+1},\dots, x_n).\label{E:partial-fermat}
\]
Clearly, any $n$-partial Fermat form is linearly equivalent to the standard Fermat form. Furthermore, 
all $k$-partial Fermat forms are direct sums.
We denote by $\DS_m$ the locus of direct sums in $S_{m}$. 

\section{Associated form of a balanced complete intersection}\label{sectionmainres}

Fix $d\ge 2$. In what follows the trivial case $(n,d)=(2,2)$ will be excluded. A length $n$ regular sequence $g_1,\dots, g_n$ of elements of $S_d$ will be called a balanced complete intersection of type $(d)^n$. 
It defines a graded Gorenstein Artin $\CC$-algebra 
\[
\cA(g_1,\dots,g_n):=S/(g_1,\dots, g_{n}),
\] whose socle lies in degree $n(d-1)$. In \cite{alper-isaev-assoc-binary} an element $\AA(g_1,\dots, g_n) \in \D_{n(d-1)}$, called
\emph{the associated form of $g_1,\dots,g_n$}, was introduced. The form $\AA(g_1,\dots, g_n)$ is a homogeneous Macaulay inverse system,
or a dual socle generator, of the algebra $\cA(g_1,\dots,g_n)$. It follows that
$[\AA(g_1,\dots, g_n)]\in \PP \D_{n(d-1)}$ depends only on the linear span 
$\langle g_1,\dots, g_n\rangle$, which we regard as a point in $\Grass(n, S_d)$. 

Recall that $g_1,\dots, g_n$ is a regular sequence in $S_d$ if and only if
$\langle g_1,\dots, g_n\rangle$ does not in lie in the resultant divisor $\frakRes \subset \Grass(n, S_d)$.
Setting $\Grass(n, S_d)_{\Res} := \Grass(n, S_d) \setminus \frakRes$, 
we obtain a morphism
\[
\AA \colon \Grass(n, S_d)_{\Res} \to \PP \D_{n(d-1)}.
\]
Given $f\in S_{d+1}$, the partial derivatives
$\partial f/\partial x_1,\dots, \partial f/\partial x_n$ form a regular sequence if and only 
if $f$ is non-degenerate. For a non-degenerate $f\in S_{d+1}$, in \cite{{alper-isaev-assoc},{eastwood-isaev1}} the \emph{associated form of $f$} was defined to be
\[
A(f):=\AA(\partial f/\partial x_1,\dots, \partial f/\partial x_n) \in \D_{n(d-1)}.
\]
Summarizing, we obtain a commutative diagram
\[
\begin{gathered}
\xymatrix{
\PP(S_{d+1})_{\Delta} \ar[rd]_{\nabla} \ar[rr]^{A} & & \PP(\D_{n(d-1)}) \\
& \Grass(n, S_d)_{\Res} \ar[ur]_{\AA}, &
}
\end{gathered}\label{D:affine-triangle}
\]
where $\PP(S_{d+1})_{\Delta}$ denotes the complement to the discriminant divisor in $\PP(S_{d+1})$ and $\nabla$ is the morphism sending a form into the linear span of its first partial derivatives. The above diagram is equivariant with respect to the standard $\SL(n)$-actions on $S$ and $\D$.
By \cite{alper-isaev-assoc-binary}, the morphism $\AA$ is a locally closed immersion, and it was proved in \cite{fedorchuk-isaev} that
$\AA$ sends polystable orbits to polystable orbits. Passing to the GIT quotients, we thus obtain
a commutative diagram
\begin{equation}\label{D:affine-triangle-GIT}
\begin{gathered}
\xymatrix{
\PP(S_{d+1})_{\Delta}\gitq \SL(n) \ar[rd]_{\widetilde{\nabla}} \ar[rr]^{\bar{A}} & & \PP(\D_{n(d-1)})^{ss} \gitq \SL(n) \\
& \Grass(n, S_d)_{\Res}\gitq  \SL(n) \ar[ur]_{\bar{\AA}}, &
}
\end{gathered}
\end{equation}
where $\widetilde{\nabla}:=\nabla\gitq \SL(n)$ is a finite injective morphism (see \cite{fedorchuk-ss}) and
 $\bar{\AA}:=\AA\gitq \SL(n)$ is a locally closed immersion. 
The main focus of this paper is the geometry of diagram \eqref{D:affine-triangle-GIT}. 

Noting that by \cite{fedorchuk-ss} the map $\nabla$ extends to a morphism from $\PP(S_{d+1})^{ss}$ to\linebreak $\Grass(n, S_d)^{ss}$ and thus induces a map $\overline{\nabla}$ of the corresponding GIT quotients, we will now state our two main results as follows:

\begin{theorem}
\label{T:gradient} 
The morphism $\overline{\nabla} \colon \PP(S_{d+1})^{ss}\gitq \SL(n) \to \Grass(n, S_d)^{ss}\gitq  \SL(n)$ is a closed immersion.
\end{theorem}

\begin{theorem}\label{T:barA1} The rational map 
\[
\bar{A}\colon \PP(S_{d+1})^{ss}\gitq \SL(n) \dashrightarrow \PP(\D_{n(d-1)})^{ss} \gitq \SL(n)
\]
extends to the generic point  
of the discriminant divisor $\Delta \gitq \SL(n)$ in the GIT compactification and contracts the discriminant divisor to a lower-dimensional variety for all sufficiently large $d$ as described in Corollaries {\rm\ref{cor:n-1,n}} and {\rm\ref{C:contraction}}.
\end{theorem}

\section{Preliminaries on dualities}

\label{S:dualities}
In this section we collect results on Macaulay inverse systems of graded
Gorenstein Artin $\CC$-algebras. We also recall the duality between the 
Hilbert points of such algebras and the gradient points of their inverse systems.

Recall that we regard $S=\CC[x_1,\dots,x_n]$ as a ring of polynomial differential operators
on the graded dual ring $\D:=\CC[z_1,\dots,z_n]$ via polar pairing \eqref{E:apolarity-action}.
For every positive $m$, the restricted pairing
\[
S_{m}  \times \D_{m} \to \CC\label{E:pairing}
\]
is perfect and so defines an isomorphism 
\begin{equation}\label{E:duality}
\D_{m} \simeq S_{m}^{\vee},
\end{equation}
where, as usual, $V^{\vee}$ stands for the dual of a vector space $V$.

Given $W\subset \D$, we define 
\[
W^{\perp}:=\{f\in S \mid f\circ g=0, \text{ for all $g\in W$}\}\subset S.\label{E:perp-of-D}
\]
Similarly given $U\subset S$, we define 
\[
U^{\perp}:=\{g\in \D \mid f\circ g=0, \text{ for all $f\in U$}\}\subset \D.\label{E:perp-of-S}
\]

\begin{claim}\label{C:explicit}
Isomorphism \eqref{E:duality} sends an element $\omega\in S_{m}^{\vee}$
to the element 
\[
\mathfrak{D}_{\omega}:=\sum_{i_1+\cdots+i_n=m} \frac{\omega\bigl(x_1^{i_1}\cdots x_n^{i_n}\bigr)}{i_1!\cdots i_n!} z_1^{i_1} \cdots z_n^{i_n}\in \D_m.
\]
Conversely, an element $g\in \D_{m}$ is mapped by isomorphism \eqref{E:duality} to the projection
\[
S_{m}\twoheadrightarrow S_m/(g^{\perp})_m \simeq \CC,
\]
where the isomorphism with $\CC$ is chosen so that $1\in \CC$ pairs to $1$ with $g$.
\end{claim}
\begin{proof}
One observes that $f\circ \mathfrak{D}_{\omega}=\omega(f)$ for every $f\in S_{m}$, and
the first part of the claim follows. The second part is immediate from definitions.
\end{proof}
\begin{corollary}\label{C:evaluation} Given $\omega\in S_{m}^{\vee}$,
for every $(a_1,\dots,a_n)\in \CC^n$ we have
\begin{equation}
\mathfrak{D}_{\omega}(a_1,\dots,a_n)=\omega\bigl((a_1x_1+\cdots+a_nx_n)^m/m!\bigr).\label{assocform}
\end{equation}
\end{corollary}
\begin{proof} 
\begin{align*} 
\omega\bigl((a_1x_1+\cdots+a_nx_n)^m/m!\bigr)
&=\frac{(a_1x_1+\cdots+a_nx_n)^m}{m!} \circ {\mathfrak D}_{\omega}  \\
&=\frac{(a_1\partial /\partial z_1+\cdots+a_n\partial /\partial z_n)^m}{m!} {\mathfrak D}_{\omega}
={\mathfrak D}_{\omega}(a_1,\dots,a_n),
\end{align*}
where the last equality is easily checked, say on monomials. 
\end{proof}

\begin{remark}\label{R:vanishing}
It follows from Corollary \ref{C:evaluation} 
that all forms in a subset $W\subset \D_m$ vanish at a given point $(a_1,\dots,a_n)\in \CC^n$ if and only if 
$(a_1x_1+\cdots+a_nx_n)^m \in W^{\perp}$.
\end{remark}

Notice that the maps
\begin{equation*}
\bigl[\langle \mathfrak{D}_{\omega}\rangle \subset \D_{m}\bigr] \mapsto \\ \bigl[(\mathfrak{D}_{\omega}^{\perp})_m\subset S_m\bigr]= 
\bigl[\ker(\omega)\subset S_m\bigr]
\end{equation*}
define isomorphisms
\[
\Grass(1, \D_{m}) \simeq \Grass\left(\dim_{\CC} S_{m}-1, S_{m}\right).
\]
More generally, for any $1\le m\le \binom{m+n-1}{n-1}-1$ the correspondence
\begin{equation*}
\bigl[W \subset \D_{m}\bigr] \mapsto \\ \bigl[ (W^{\perp})_m\subset S_m\bigr]
\end{equation*}
yields an isomorphism
\begin{equation}\label{E:grassmann}
\Grass(k, \D_{m}) \simeq \Grass\left(\dim_{\CC} S_{m}-k,
S_{m}\right).
\end{equation}


Let $I\subset S$ be a Gorenstein ideal and $\nu$ the socle degree of the algebra $\cA=S/I$. Recall that a \emph{(homogeneous) Macaulay inverse system} of $\cA$ is an element $f_{\A}\in \D_{\nu}$
such that 
\[
f_\A^{\perp}=I
\]
(see \cite[Lemma 2.12]{iarrobino-kanev} or \cite[Exercise 21.7]{eisenbud}).
As $(f_\A^{\perp})_{\nu}=I_{\nu}$, we see that all Macaulay inverse systems are mutually proportional and
$\langle f_{\A}\rangle=\bigl((I_{\nu})^\perp\bigr)_{\nu}$. 
Clearly, the line $\langle f_{\A}\rangle\in \Grass(1, \D_{\nu})$
maps to the $\nu^{th}$ Hilbert point $H_{\nu}\in \Grass(\dim_{\CC} S_{\nu}-1, S_{\nu})$ of $\cA$ under isomorphism \eqref{E:grassmann} with $k=1$.

\begin{remark}
Papers \cite{eastwood-isaev1,eastwood-isaev2}, for any $\omega\in S_{\nu}^{\vee}$ with $\ker\omega=I_{\nu}$, introduced the {\it associated form}\, of $\cA$ as the element of $\D_{\nu}$ given by the right-hand side of formula (\ref{assocform}) with $m=\nu$ (up to the factor $\nu!$). By Corollary \ref{C:evaluation}, under isomorphism \eqref{E:grassmann} with $k=1$ the span of every associated form in $\D_{\nu}$ also maps to the $\nu^{th}$ Hilbert point $H_{\nu}\in \Grass(\dim_{\CC} S_{\nu}-1, S_{\nu})$ of $\cA$. In particular, for the algebra $\cA$ any associated form is simply one of its Macaulay inverse systems,
and equation \eqref{assocform} with $m=\nu$ and $\ker\omega=I_{\nu}$ is an explicit formula 
for a Macaulay inverse system of $\cA$ (see \cite{isaev-criterion} for more details).
\end{remark}

\subsection{Gradient points}
\label{S:gradient}
Given a polynomial $F\in \D_m$, we define the {\it $p^{th}$ gradient point of $F$}
to be the linear span of all $p^{th}$ partial derivatives of $F$ in $\D_{m-p}$. We denote
the $p^{th}$ gradient point by $\nabla^{p}(F)$. Note that 
\[
\nabla^{p}(F)=\{g \circ F \mid g\in S_{p}\}
\]
is simply the $(m-p)^{th}$ graded piece of the principal $S$-module $SF$.
The $1^{st}$ gradient point
$\nabla F:=\langle \partial F/\partial z_1, \dots, \partial F /\partial z_n\rangle$
will be called simply \emph{the gradient point of $F$}.

\begin{prop}[Duality between gradient and Hilbert points]\label{P:hilb-dual}
The $p^{th}$ gradient point of a Macaulay inverse system $f_{\A}\in \D_{\nu}$ 
maps to the $(\nu-p)^{th}$ Hilbert point $H_{\nu-p}$ of $\cA$
under isomorphism \eqref{E:grassmann}.
\end{prop}

\begin{proof} Let $G$ be the $p^{th}$ gradient point of $f_{\A}$, that is \[G:=\left\langle \frac{\partial^p}{\partial z_1^{i_1} \cdots
\partial z_n^{i_n}} f_{\A} \mid i_1+\cdots+i_n=p\right\rangle.\] We need to verify that $I_{\nu-p}=(G^{\perp})_{\nu-p}$. We have
\begin{align*}
(G^{\perp})_{\nu-p}& =\left\{f\in S_{\nu-p} \mid f\circ \frac{\partial^p}{\partial z_1^{i_1}
 \cdots \partial z_n^{i_n}} f_{\A} =0 \text{ for all $i_1+\cdots+i_n=p$}\right\}\\
&=\left\{f\in S_{\nu-p} \mid f x_1^{i_1}\cdots x_n^{i_n} \circ f_{\cA} =0 \text{ for all degree $p$ monomials}\right\} \\
&=\left\{f\in S_{\nu-p} \mid x_1^{i_1}\cdots x_n^{i_n} f \in f_\cA^{\perp} \text{ for all degree $p$ monomials}\right\} \\
&=\left\{f\in S_{\nu-p} \mid x_1^{i_1}\cdots x_n^{i_n}f  \in I_{\nu} \text{ for all degree $p$ monomials}\right\} \\
&=I_{\nu-p},
\end{align*}
where the last equality comes from the fact that $I$ is Gorenstein. 
\end{proof}

As a corollary of the above duality result, we recall in Proposition \ref{P:singular} below
a generalization of \cite[Lemma 4.4]{alper-isaev-assoc}. Although this statement is well-known
(it appears, for example, in \cite[Proposition 4.1, p.~174]{rota}), we provide a short
proof for completeness. 
We first recall that a non-zero homogeneous form $f$ in $n$ variables 
has multiplicity $\ell+1$ at a point $p\in \PP^{n-1}$ if and only if
 all partial derivatives of $f$ of order $\ell$ (hence of all orders $\leq \ell$) 
 vanish at $p$, and some partial derivative of $f$ of order $\ell+1$ does not vanish at $p$.  
We define the \emph{Veronese cone} $\mathcal{C}_m$ to be the variety of all degree $m$ powers of linear forms in $S_m$:
\[
\mathcal{C}_{m}:=\bigl\{L^{m} \mid L \in  S_1\bigr\}\subset S_{m}.
\]
\begin{prop}\label{P:singular} 
Let $I\subset S$ be a Gorenstein ideal and $\nu$ the socle degree of the algebra $\cA=S/I$. Then a Macaulay inverse system $f_{\cA}$ of $\cA$ has a point of multiplicity $\ell+1$ 
if and only if there exists a non-zero $L\in S_1$ 
such that $L^{\nu-\ell}\in I_{\nu-\ell}$, 
and $L^{\nu-\ell-1} \not\in  I_{\nu-\ell-1}$. In particular, 
$f_{\cA}$ has no points of multiplicity $\ell+1$ or higher if and only if
\[
I_{\nu-\ell} \cap \mathcal{C}_{\nu-\ell} = (0).
\]
\end{prop}

\begin{proof} By Proposition \ref{P:hilb-dual}, the $\ell^{th}$ gradient point of $f_{\cA}$ is dual 
to the $(\nu-\ell)^{th}$ Hilbert point of $\cA$
\[
H_{\nu-\ell}\colon S_{\nu-\ell} \twohead \cA_{\nu-\ell}.\]
We conclude by Remark \ref{R:vanishing} 
that all partial derivatives of $f_{\cA}$ of order $\ell$ vanish at $(a_1,\dots,a_n)$ if and only if
\[
(a_1x_1+\cdots+a_nx_n)^{\nu-\ell}\in \ker H_{\nu-\ell}=I_{\nu-\ell}.
\]
It follows that $L=a_1x_1+\cdots+a_nx_n$ satisfies 
$L^{\nu-\ell}\in I_{\nu-\ell}$
and $L^{\nu-\ell-1} \not\in  I_{\nu-\ell-1}$ if and only if $f_{\cA}$ has multiplicity exactly $\ell+1$
at the point $(a_1,\dots, a_n)$. 
\end{proof}

\section{The gradient morphism $\nabla$}

In this section, we prove Theorem \ref{T:gradient}.  Recall that we have the commutative diagram
\[
\begin{aligned}
\xymatrix{ 
\PP\bigl(S_{d+1})^{ss} \ar[d]^{\pi_0}  \ar[r]^{\nabla}  & \Grass\bigl(n, S_d\bigr)^{ss}\ar[d]^{\pi_1} \\
\PP\bigl(S_{d+1})^{ss}\gitq \SL(n) \ar[r]^{\overline{\nabla}\qquad} 
&\Grass\bigl(n, S_d\bigr)^{ss} \gitq \SL(n).
}
\end{aligned}\label{diagram}
\]
Let $\DS^{ss}_{d+1}:=\PP(\DS_{d+1})^{ss}$ be the locus of semistable direct sums in $\PP\bigl(S_{d+1})^{ss}$.
By \cite[Section 3]{fedorchuk-direct}, the set $\DS^{ss}_{d+1}$ is precisely the closed
locus in $\PP\bigl(S_{d+1})^{ss}$ where $\nabla$ has positive fiber dimension.   

Suppose $f\in S_{d+1}$ is a semistable form. 
Then, after a linear change of variables, we have a maximally fine direct sum decomposition
\begin{equation}
\label{E:direct-sum}
f=\sum_{i=1}^k f_i (\mathbf{x}^i),
\end{equation}
where $V_i=\langle \mathbf{x}^i\rangle$ are such that $V=\oplus_{i=1}^k V_i$, and where each $f_i$ is not a direct sum in $\Sym V_i$.
Set $n_i:=\dim_{\CC} V_i$. 
We define the canonical torus $\Theta(f)\subset \SL(n)$ associated to $f$ as the connected component of the identity of the subgroup
\[
\{g\in \SL(n) \mid \text{$V_i$ is an eigenspace of $g$, for every $i=1,\dots, k$}\}\subset\SL(n).\label{E:torus}
\]
Clearly, $\Theta(f) \simeq (\CC^{*})^{k-1}$, and since  
\[
\nabla([f])=\nabla([f_1])\oplus \cdots \oplus \nabla([f_k]), \ \text{where $\nabla([f_i]) \in \Grass(n_i, \Sym^{d} V_i)$},
\]
we also have $\Theta(f)\subset \Stab(\nabla([f]))$, where $\Stab$ denotes the stabilizer under the $\SL(n)$-action.

From the definition of $\Theta(f)$, it is clear that $\Theta(f)\cdot [f] \subset \nabla^{-1}(\nabla([f]))$,
and in fact \cite[Corollary 3.12]{fedorchuk-direct} gives a set-theoretic equality 
 $\nabla^{-1}(\nabla([f]))=\Theta(f)\cdot [f]$. We will now obtain a stronger result:
\begin{lemma}
\label{L:gradient-fiber}
One has $\nabla^{-1}(\nabla([f]))=\Theta(f)\cdot [f]$ scheme-theoretically, or, equivalently,
\[
\ker(d \nabla_{[f]})={\mathbf T}_{[f]}(\Theta(f)\cdot [f]),
\]
where ${\mathbf T}_{[f]}$ denotes the tangent space at $[f]$.
\end{lemma}
\begin{proof}
Under the standard identification of ${\mathbf T}_{[f]} \PP(S_{d+1})$ with $S_{d+1}/\langle f\rangle$, the subspace
${\mathbf T}_{[f]}(\Theta(f)\cdot [f])$ is identified with $\langle f_1, \dots, f_{k}\rangle/\langle f\rangle$.  
It now suffices to show that every 
$g\in S_{d+1}$ that satisfies $\nabla[g] \subset \nabla[f]$ must lie in $\langle f_1, \dots, f_{k}\rangle$, where\linebreak $\nabla[g]:=\langle \partial g/\partial x_1, \dots, \partial g/\partial x_n\rangle\subset S_d$. This is precisely
the statement of \cite[Corollary 3.12]{fedorchuk-direct}.
\end{proof}

We note an immediate consequence:
\begin{corollary} If $f\in S_{d+1}^{ss}$ is not a direct sum, then $\nabla$ is unramified at $[f]$.
\end{corollary}

Further, since $\nabla$ is equivariant with respect to the $\SL(n)$-action, we 
have the inclusion $\Stab([f])\subset \Stab(\nabla([f]))$. 
As the following result shows, the difference between $\Stab([f])$ and $\Stab(\nabla([f]))$ 
is controlled by the torus $\Theta(f)$.

\begin{corollary}\label{C:gradient-stabilizer} The subgroup $\Stab(\nabla([f]))$ is generated by $\Theta(f)$ and $\Stab([f])$. 
\end{corollary}
\begin{proof}
Suppose $\sigma\in \Stab(\nabla([f]))$. Then $\nabla(\sigma\cdot[f])=\nabla([f])$ implies by Lemma 
\ref{L:gradient-fiber} that $\sigma\cdot[f]=\tau \cdot[f]$ for some $\tau \in \Theta(f)$. Consequently,
$\tau^{-1}\circ \sigma \in \Stab([f])$ as desired. 
\end{proof}

Next, we obtain the following generalization of \cite[Proposition 6.3]{alper-isaev-assoc-binary}, whose proof we follow almost verbatim.
\begin{prop} 
The morphism $\nabla$ is a closed immersion along the open locus $\cU:=\PP(S_{d+1})^{ss} \setminus \DS^{ss}_{d+1}$ of all elements that are not direct sums.
\end{prop}
\begin{proof}
Since for every $[f]\in \cU$ we have that $\nabla$ is unramified at $[f]$ and $\nabla^{-1}(\nabla([f]))=[f]$,
it suffices to show that $\nabla$ is a finite morphism when restricted to $\cU$. 
Since, by \cite{fedorchuk-ss}, the induced morphism on the GIT quotients is finite, by \cite[p.~89, Lemme]{luna-slices-etales}
it suffices to verify that $\nabla$ is quasi-finite and that $\nabla$ sends closed orbits to closed orbits. The former 
has already been established, and the latter is proved below in Proposition \ref{P:polystability}.
\end{proof}

\begin{prop}\label{P:polystability}
Suppose $f\in S_{d+1}^{ss}$ is polystable and not a direct sum. Then the image
$\nabla([f])\in \Grass(n,S_d)^{ss}$ is polystable.
\end{prop}
The above result is a generalization of \cite[Theorem 1.1]{fedorchuk-ss},
whose method of proof we follow; we also keep the notation of \emph{loc.cit.}, especially
as it relates to monomial orderings.
We begin with a preliminary observation.
\begin{lemma}\label{1-PS-DS}
Suppose $f\in S_{d+1}$ is such that there exists a non-trivial one-parameter subgroup
$\lambda$ of $\SL(n)$ acting diagonally on $x_1,\dots, x_n$ with weights $\lambda_1,\dots, \lambda_n$ and 
satisfying 
\[
w_{\lambda}(\init_{\lambda}(\partial f/\partial x_i))=d\lambda_i.
\]
Then $f$ is a direct sum.
\end{lemma}
\begin{proof} We can assume that
\[
\lambda_1\leq \cdots \leq \lambda_{a}<\lambda_{a+1}=\cdots=\lambda_n
\]
for some $1\le a<n$. Then the fact that \[
w_{\lambda}(\init_{\lambda}(\partial f/\partial x_i))=d\lambda_i=d\lambda_n,
\]
for all $i=a+1,\dots, n$, implies 
\[
\partial f/\partial x_{a+1}, \dots, \partial f/\partial x_{n}\in \CC[x_{a+1},\dots, x_n].
\]
Consequently, $f=g_1(x_1,\dots,x_a)+g_2(x_{a+1},\dots, x_n)$ is a direct sum.
\end{proof}

\begin{proof}[Proof of Proposition {\rm\ref{P:polystability}}] Since $f$ is polystable, by \cite[Theorem 1.1]{fedorchuk-ss} it follows that $\nabla([f])$ is semistable. Suppose $\nabla([f])$ is not polystable. Then there exists a one-parameter
subgroup $\lambda$ acting on the coordinates $x_1,\dots, x_n$ with the weights $\lambda_1, \dots, \lambda_n$
such that the limit of $\nabla([f])$ under $\lambda$ exists 
and does not lie in the orbit of $\nabla([f])$. In particular, the limit of $[f]$ under $\lambda$ does not exist.

Then by \cite[Lemma 3.5]{fedorchuk-ss}, there is 
an upper triangular unipotent coordinate change 
\[\label{E:sub}
\begin{aligned}
x_1& \mapsto x_1+c_{12}x_2+\cdots+ c_{1n}x_n, \\
x_2&\mapsto\phantom{{}=1111} x_2+\cdots+ c_{2n}x_n, \\
\vdots \\
x_n&\mapsto \phantom{{}=1111111111111111}x_n
\end{aligned}
\]
such that for the transformed form
$$
h(x_1,\dots,x_n):=f(x_1+c_{12}x_2+\cdots+ c_{1n}x_n, x_2+\cdots+ c_{2n}x_n, \dots, x_n)
$$
the initial monomials
$$
\init_{\lambda}(\partial h/\partial x_1), \dots, \init_{\lambda}(\partial h/\partial x_n)
$$ 
are distinct. Now, setting 
$$
\mu_i:=w_{\lambda}(\init_{\lambda}(\partial h/\partial x_i)),
$$
by \cite[Lemma 3.2]{fedorchuk-ss} we have 
$$
\mu_1+\cdots+\mu_n= 0.
$$ 
It follows that with the respect to the one-parameter subgroup 
$\lambda'$ acting on $x_i$ with the weight $d\lambda_i-\mu_i$,
all monomials of $h$ have non-negative weights (cf. \cite[the proof of Lemma 3.6]{fedorchuk-ss}). 
Write $h=h_0+h_1$, where all monomials of $h_0$ have 
zero $\lambda'$-weights and all monomials of $h_1$ have positive $\lambda'$-weights. 
Then $h_0\in \overline{\SL(n)\cdot h}=\SL(n)\cdot h$, 
by the polystability assumption on $f$. Furthermore, $h_0$ is stabilized by $\lambda'$. 

If $\lambda'$ is a trivial one-parameter subgroup, then $\mu_i=d\lambda_i$ for all $i=1,\dots, n$,
and by Lemma \ref{1-PS-DS} the form $h$ is a direct sum, which is a contradiction.

Suppose now that $\lambda'$ is a non-trivial one-parameter subgroup.
Clearly, we have
\[
w_{\lambda}(\init_{\lambda}(\partial h_0/\partial x_i)\geq  w_{\lambda}(\init_{\lambda}(\partial h/\partial x_i),
\]
since the state of $h_0$ is a subset of the state of $h$. If one of the inequalities above
is strict, then $\nabla([h_0])$ is destabilized by $\lambda$, contradicting the semistability of $\nabla([h_0])$
established in \cite[Theorem 1.1]{fedorchuk-ss}. Thus 
\[
w_{\lambda}(\init_{\lambda}(\partial h_0/\partial x_i))=w_{\lambda}(\init_{\lambda}(\partial h/\partial x_i))=\mu_i.
\]
Moreover, since $h_0$ is $\lambda'$-invariant, we have that $\partial h_0/\partial x_i$
is homogeneous of degree $-w_{\lambda'}(x_i)=\mu_i-d\lambda_i$ with respect to $\lambda'$. 
Let $\mu$ be the one-parameter subgroup acting on $x_1,\dots, x_n$ with the weights
$\mu_1,\dots, \mu_n$. 
It follows
that 
\[
w_{\mu}(\init_{\mu}(\partial h_0/\partial x_i))=dw_{\lambda}(\init_{\lambda}(\partial h_0/\partial x_i)+w_{\lambda'}(\init_{\lambda'}(\partial h_0/\partial x_i)=d\mu_i-\mu_i+d\lambda_i.
\]
Then the one-parameter subgroup $\lambda+\mu$ acting on $x_1,\dots, x_n$ with the weights 
$\lambda_1+\mu_1,\dots, \lambda_n+\mu_n$ satisfies
\[
\begin{array}{l}
w_{\lambda+\mu}(\init_{\lambda+\mu}(\partial h_0/\partial x_i))=
w_{\lambda}(\init_{\lambda}(\partial h_0/\partial x_i))+w_{\mu}(\init_{\mu}(\partial h_0/\partial x_i))
=\\
\hspace{8cm}d\mu_i-\mu_i+d\lambda_i+\mu_i=d(\mu_i+\lambda_i).
\end{array}
\]
Applying Lemma \ref{1-PS-DS}, we conclude that either $h_0$ is a direct sum, or
\[
\lambda_i+\mu_i=0\quad  \text{for all $i=1,\dots, n$}.
\]
In the latter case,
it follows that $\lambda$ is proportional to $\lambda'=d\lambda-\mu$. Since the limit of $h$ under $\lambda'$
exists and is equal to $h_0$, the limit under $\lambda$ of $h$ must exist and be equal to $h_0$ as well. Observing that the inverse of an upper-triangular matrix with 1's on the diagonal has the same form, we see that the limit of 
$$
f(x_1,\dots,x_n)=h(x_1+c_{12}'x_2+\cdots+ c_{1n}'x_n, x_2+\cdots+ c_{2n}'x_n, \dots, x_n)
$$
under $\lambda$ also exists. This contradiction concludes the proof.
\end{proof}

\begin{corollary}\label{C:nablaprespolystability}
The morphism $\nabla\colon \PP(S_{d+1})^{ss} \to \Grass(n,S_d)^{ss}$ preserves polystability.
\end{corollary}
\begin{proof}
Suppose $f=f_1+\cdots+f_k$ is the maximally fine direct sum decomposition of a polystable
form $f$, where $f_i\in \Sym^{d+1} V_i$, and where 
$V=\oplus_{i=1}^k V_i$. Then each $f_i$ is polystable and not a direct sum in $\Sym^{d+1} V_i$. Hence 
$\nabla([f_i])$ is polystable with respect to the $\SL(V_i)$-action. 

Since $\Theta(f)\subset \Stab(\nabla([f]))$ is a reductive subgroup, 
to prove that $\nabla([f])$ is polystable, it suffices to verify that $\nabla([f])$ is polystable with respect
to the centralizer $C_{\SL(n)}(\Theta(f))$ of $\Stab(\Theta(f))$ in $\SL(n)$, see \cite[Corollaire 1 and Remarque 1]{luna-adherences}. 
We have
\[
C_{\SL(n)}(\Theta(f))=\left(\GL(V_1)\times \cdots \times \GL(V_k)\right)\cap \SL(n).
\]
Arguing as on \cite[p.~456]{fedorchuk-ss}, we see that every one-parameter subgroup $\lambda$ of $C_{\SL(n)}(\Theta(f))$ can be renormalized
to a one-parameter subgroup of $\SL(V_1)\times \cdots \times \SL(V_k)$ without changing its action on $\nabla([f])$.
Since $\nabla([f_i])$ is polystable with respect to $\SL(V_i)$, it follows that
\[
\nabla([f])=\nabla([f_1])\oplus \cdots \oplus \nabla([f_k])
\] 
is polystable with respect to the action of $\lambda$ thus proving the claim.\end{proof}

\begin{proof}[Proof of Theorem {\rm\ref{T:gradient}}] 
Suppose that $f$ is polystable, consider its maximally fine direct sum decomposition and the canonical torus $\Theta(f)$ in $\Stab(\nabla([f]))$ as constructed above. In what follows, we will write $X$ to denote $\PP(S_{d+1})^{ss}$ and $Y$ to denote $\Grass(n, S_d)^{ss}$. Set $p:=\pi_0([f])\in X\gitq \SL(n)$.

We will prove that $\overline{\nabla}$ is unramified at $p$. 
Let  $N_{[f]}$ 
be the normal space to the $\SL(n)$-orbit of $[f]$ in $X$ at the point $[f]$, and $N_{\nabla([f])}$ 
the normal space to the $\SL(n)$-orbit of $\nabla([f])$ in $Y$ at the point $\nabla([f])$.
We have a natural map  
\[
\iota\colon N_{[f]} \to N_{\nabla([f])}
\]
induced by the differential of $\nabla$. The map $\iota$ is injective by Lemma \ref{L:gradient-fiber}.
 
Since both $[f]$ and $\nabla([f])$ have closed orbits in $X$ and $Y$, respectively (see Corollary \ref{C:nablaprespolystability}), 
to verify that $\overline{\nabla}$ is unramified at $p$, it suffices, 
by Luna's \'etale slice theorem, to prove that the morphism
\begin{equation}
\label{E:slice}
s(f)\colon N_{[f]} \gitq \Stab([f]) \to N_{\nabla([f])}\gitq \Stab(\nabla([f])
\end{equation}
is unramified.  

As $\nabla$ is not necessarily stabilizer-preserving at $[f]$ (i.e., $\Stab([f])$ may not be equal to $\Stab(\nabla([f]))$), 
we cannot directly appeal to the injectivity of $\iota$. Instead, consider the $\Theta(f)$-orbit, say $F$, of $[f]$ in $X$. Let $\cN_{F/X}$ be the $\Theta(f)$-invariant normal bundle of $F$ in $X$. Since by Lemma \ref{L:gradient-fiber} we have $\nabla^{-1}(\nabla([f]))=F$,
there is a natural $\Theta(f)$-equivariant map $J:\cN_{F/X} \to  N_{\nabla([f])}$. 
We now make a key observation that for the induced map $\tilde J:\cN_{F/X} \gitq \Theta(f)\to N_{\nabla([f])}$ one has
\[
\tilde J(\cN_{F/X} \gitq \Theta(f))=\iota\left(N_{[f]} \right).
\]

Since $\overline{\nabla}$ is finite by \cite[Proposition 2.1]{fedorchuk-ss}, 
the morphism $s(f)$ from Equation \eqref{E:slice} is quasi-finite.  
Applying Lemma \ref{L:GIT-lemma} (proved below), with $\spec A=N_{[f]}$, $\spec B=N_{\nabla([f])}$,
$T=\Theta(f)$, $H=\Stab([f])$, $G=\Stab(\nabla([f]))$, as well as Corollary \ref{C:gradient-stabilizer}, 
we obtain that $s(f)$ is in fact a closed immersion, and so is unramified. Note that here the group $G$ is reductive by Matsushima's criterion. This proves that $\overline{\nabla}$ is unramified at $p$.

We now note that $\overline{\nabla}$ is injective. Indeed, this follows as in the proof of \cite[Part (2) of Proposition 2.1]{fedorchuk-ss} from Corollary \ref{C:nablaprespolystability} and the finiteness of $\overline{\nabla}$. We then conclude that $\overline{\nabla}$ is a closed immersion.
\end{proof}

\begin{lemma}[GIT lemma] 
\label{L:GIT-lemma} Suppose $G$ is a reductive group. Suppose $T\subset G$ is a connected 
reductive subgroup, and $H\subset G$ is a reductive subgroup 
such that $G$ is generated by $T$ and $H$.   Suppose we have a $G$-equivariant closed immersion 
of normal affine schemes admitting an action of $G$
\[
\spec A \hookrightarrow \spec B.
\]
such that $\spec A^H \to \spec B^G$ is quasi-finite.  Then $\spec A^G\simeq \spec A^H$ and, consequently,
 $\spec A^H \to \spec B^G$ is a closed immersion.
 \end{lemma}
 \begin{proof}
 We have the following commutative diagram
\[
 \xymatrix{
 \spec A^H \ar[d] \ar[rd] \ar@{^(->}[r] & \spec B^H \ar[d] \\
  (\spec A^H)\gitq T\simeq \spec A^G \ar@{^(->}[r]  & (\spec B^H)\gitq T \simeq \spec B^G.
  }
\]
Since the diagonal arrow is quasi-finite by assumption, and the bottom arrow is a closed immersion, we conclude
that the GIT quotient $\spec A^H \to (\spec A^H)\gitq T$ is quasi-finite as well.  Since this is a good quotient by a connected
group, the morphism $\spec A^H \to (\spec A^H)\gitq T\simeq \spec A^G$ must be an isomorphism.  
 \end{proof}
 
\begin{corollary}[Theorem \ref{T:barA}]\label{corT:A}
The morphism 
\[
\bar{A}\colon \PP(S_{d+1})_{\Delta}\gitq \SL(n) \rightarrow \PP(\D_{n(d-1)})^{ss} \gitq \SL(n)
\]
is a locally closed immersion.
\end{corollary} 

\section{The morphism $\AA_{\Gr}$}

In this section, we prove Theorem \ref{T:barA1}. In fact, we study in detail the rational map $\bar A \colon (\PP S_{d+1})^{ss}\gitq \SL(n) \dashrightarrow \PP(\D_{n(d-1)})^{ss}\gitq \SL(n)$ in codimension one. 

As in Section \ref{sectionmainres}, fix $d\ge 2$. As always, we assume that $n\ge 2$ and disregard the trivial case $(n,d)=(2,2)$. Given $U\in \Grass(n, S_d)$, we take $I_U$ to be the ideal in $S$ generated by the elements in $U$. Consider the following locus in $\Grass(n, S_d)$:
\[
W_{n,d} =\{U\in \Grass(n, S_d) \mid \dim_{\CC} (S/I_U)_{n(d-1)-1} = n\}.\label{E:Wnd}
\]
Since $\dim_{\CC}  (S/I_U)_{n(d-1)-1}$ is an upper semi-continuous function on $\Grass(n, S_d)$ and 
for every $U\in \Grass(n, S_d)$ one has $\dim_{\CC} (S/I_U)_{n(d-1)-1} \geq n$, 
we conclude that $W_{n,d}$ is an open subset of $\Grass(n, S_d)$. Moreover,
since for $U\in\Grass(n, S_d)_{\Res}$ the ideal $I_U$ is Gorenstein of socle degree $n(d-1)$, we have
$\Grass(n, S_d)_{\Res} \subset W_{n,d}$. 

Applying polar pairing, we obtain a morphism 
\[
\begin{aligned}
\AA_{\Gr} \colon W_{n,d} \to \Grass(n, \D_{n(d-1)-1}),  \\
\quad \AA_{\Gr}(U)=\left[(I_U)_{n(d-1)-1}^\perp \subset 
\D_{n(d-1)-1}\right].
\end{aligned}\label{E:A-Gr}
\]

From the duality between Hilbert and gradient points it follows that
\[
\nabla(\AA(U))=\AA_{\Gr}(U) \text{ for
every $U\in \Grass(n, S_d)_{\Res}$.}
\]  

We conclude that we have the commutative diagram:
\[
\makebox[250pt]{$\begin{gathered}
\xymatrix{
\PP(S_{d+1})^{ss}\gitq \SL(n)&&\PP(\D_{n(d-1)})^{ss}\gitq \SL(n)\\
\PP(S_{d+1})^{ss} \ \ar[u]^{\pi_0} \ar[d]_{\nabla}   & \PP(S_{d+1})_{\Delta} \ar@{_{(}->}[l] \ar[r]^{A} \ar[d]^{\nabla}  &  \PP(\D_{n(d-1)})^{ss} \ar[d]^{\nabla} \ar[u]_{\pi_2}  \\
\Grass(n, S_d)^{ss} \ \ar[d]_{\pi_1} 
& \Grass(n, S_d)_{\Res} \ar@{_{(}->}[l] \ar@{_{(}->}[d]   \ar[ur]^{\AA} \ar[r] &  \Grass(n, \D_{n(d-1)-1})^{ss} \ \ar[d]^{\pi_3} \\
\Grass(n, S_d)^{ss}\gitq \SL(n) & W_{n,d} \ar[ru]^{\AA_{\Gr}} &  \Grass(n, \D_{n(d-1)-1})^{ss}\gitq \SL(n).
}
\end{gathered}$}
\]

\begin{prop}\label{P:A-Gr-defined} Suppose $U\in \Grass(n, S_{d})$ is such that
\[
\VV(I_U)=\{p_1,\dots,p_k\}
\]
is scheme-theoretically a set of $k$ distinct points in general linear position in $\PP^{n-1}$. Then 
$U\in W_{n,d}$. 
\end{prop}

\begin{remark}
A set $\{p_1,\dots,p_k\}$ points in $\PP^{n-1}$ is in general linear position if and only if $k\leq n$, and,
up to the $\PGL(n)$-action,
\begin{align*}
p_i=\{x_1=\cdots=\widehat{x_i}=\cdots=x_n=0\},\quad i=1,\dots, k,
\end{align*}
in the homogeneous coordinates $[x_1:\dots:x_n]$ on $\PP^{n-1}$. 
\end{remark}

\begin{proof}[Proof of Proposition {\rm\ref{P:A-Gr-defined}}] 

Since $\operatorname{depth}(I_U)=n-1$, we can choose degree $d$ generators $g_1,\dots, g_n$ of $I_U$ such that 
$g_1,\dots, g_{n-1}$ form a regular sequence. Then 
$\Gamma:=\VV(g_1,\dots,g_{n-1})$ is a finite-dimensional subscheme of $\PP^{n-1}$. 
By B\'ezout's theorem, $\Gamma$ is a set of $d^{n-1}$
points, counted with multiplicities.  

Set $R:=S/(g_1,\dots,g_{n-1})$. 
Consider the Koszul complex $K_{\bullet}:=K_{\bullet}(g_1,\dots, g_n)$.
We have 
\[
\HH_0(K_{\bullet})=S/(g_1,\dots,g_n)=S/I_U. 
\]
Since $g_1,\dots,g_{n-1}$ is a regular
sequence, we also have
\begin{align*}
\HH_i(K_{\bullet})= 0 \quad \text{for all $i>0$}
\end{align*}
and 
 \begin{align*}
\HH_{1}(K_{\bullet})=\bigl(((g_1,\dots,g_{n-1}):_{\, S}\hspace{-0.1cm}(g_1,\dots, g_n))/(g_1,\dots,g_{n-1})\bigr)(-d)\simeq \Ann_{R}(g_n)(-d).
\end{align*}

To establish the identity 
\[
\codim\bigl((I_U)_{n(d-1)-1}, S_{n(d-1)-1}\bigr)=n 
\]
it suffices to prove \[
\HH_{1}(K_{\bullet})_{n(d-1)-1}=0.
\]
Indeed, in this case the graded degree $n(d-1)-1$ part of the Koszul complex will be an exact complex
of vector spaces and so the dimension 
of $\left(S/I_U\right)_{n(d-1)-1}$ will coincide with that in the 
situation when $g_1,\dots,g_n$ is a regular sequence, that is, with $n$.  

As we have already observed, we have 
\[
\HH_{1}(K_{\bullet})_{n(d-1)-1}=\Ann_R(g_n)_{n(d-1)-1}(-d)=\Ann_R(g_n)_{n(d-1)-1-d}.
\]
Hence it suffices to prove that $\Ann_R(g_n)_{n(d-1)-1-d}=0$. 
Write $\Gamma=\Gamma'\cup \Gamma''$, where $\Gamma'\neq\varnothing$ and $\Gamma'':=\{p_1,\dots,p_k\}$. 
Since $g_n$ vanishes on all of $\Gamma''$ but does not vanish at any point of $\Gamma'$, 
every element of $\Ann_{R}(g_n)_{n(d-1)-1-d}$ comes from 
a degree $n(d-1)-1-d$ form that vanishes on all of $\Gamma'$.
We apply the Cayley-Bacharach Theorem \cite[Theorem CB6]{cayley-bacharach},
which implies the following statement:

\begin{claim}\label{C:CB}
Set $s:=d(n-1)-(n-1)-1=n(d-1)-d$. If $r \leq s$ is a non-negative
integer, then the dimension of the family of projective hypersurfaces 
of degree $r$ containing $\Gamma'$ modulo those containing all of $\Gamma$ 
is equal to the failure of $\Gamma''$ 
to impose independent conditions on projective hypersurfaces of complementary degree $s-r$.
\end{claim}

In our situation $r=s-1$, and $\Gamma''$ imposes independent conditions on hyperplanes
by the general linear position assumption. Hence we conclude by Claim \ref{C:CB} that 
every form of degree $n(d-1)-1-d$ that vanishes on all of $\Gamma'$ also vanishes on all of $\Gamma''$ and therefore, as the ideal $(g_1,\dots,g_{n-1})$ is saturated, maps to $0$ in $R$. We thus see that $\Ann_{R}(g_n)_{n(d-1)-1-d}=0$. 
This finishes the proof.
\end{proof}

Motivated by the result above, 
we consider the following partial stratification of the resultant divisor $\frakRes \subset \Grass(n, S_d)$. 
For $1\leq k \leq n$, define $Z_{k}$ to be the locally closed subset of $\Grass(n, S_d)$ consisting 
of all subspaces $U$ such that $\VV(I_U)$ is scheme-theoretically a set of $k$ distinct points in general linear position in $\PP^{n-1}$. 
Clearly, $Z_1$ is dense in $\frakRes$, and 
\[
\overline{Z}_{k} \supset Z_{k+1}\cup \cdots \cup Z_n. 
\]
We will also set $\Sigma_{k}:=\nabla^{-1}(Z_k) \subset \PP(S_{d+1})$. By the Jacobian criterion, $\Sigma_k$ is the locus
of hypersurfaces with only $k$ ordinary double points in general linear position and no other singularities.

\begin{lemma}\label{L:non-empty} For every $1\leq k\leq n$, one has that $Z_k$ is a non-empty and irreducible subset of 
$\Grass(n, S_d)$, 
and $\Sigma_k$ is a non-empty and irreducible subset of $\PP(S_{d+1})^{ss}$.
\end{lemma}
\begin{proof}  
It follows from the Hilbert-Mumford numerical criterion that any hypersurface in $\PP^{n-1}$ of degree
$d+1$ with at worst ordinary double point singularities is semistable. 

Having $k$ singularities at $k$ fixed points $p_1,\dots, p_k$
(resp., having $k$ fixed base points $p_1,\dots, p_k$)  in general linear position
is a linear condition on the elements of $\PP(S_{d+1})$ (resp., the elements of the Stiefel variety 
over $\Grass(n, S_d)$) and so defines an irreducible closed subvariety $\Sigma(p_1,\dots,p_k)$ 
in $\PP(S_{d+1})$ (resp., $Z(p_1,\dots,p_k)$ in $\Grass(n, S_d)$).
The property of having exactly ordinary double points at $p_1,\dots, p_k$ (resp., having the base
locus being equal to $\{p_1,\dots, p_k\}$ scheme-theoretically) is an open condition in 
$\Sigma(p_1,\dots,p_k)$ in $\PP(S_{d+1})$ (resp., $Z(p_1,\dots,p_k)$ in $\Grass(n, S_d)$) and so defines an irreducible
subvariety $\Sigma^{0}(p_1,\dots,p_k)$ (resp., $Z^0(p_1,\dots,p_k)$). We conclude
the proof of irreducibility by noting that $\Sigma_k=\PGL(n)\cdot \Sigma^{0}(p_1,\dots,p_k)$
(resp., $Z_k=\PGL(n)\cdot Z^{0}(p_1,\dots,p_k)$).

Since $\Sigma_{k}=\nabla^{-1}(Z_k)$, it suffices to check the non-emptiness of $\Sigma_{k}$.  
If $F\in \Sigma_n$ has ordinary double points at $p_1,\dots, p_n$, then by the deformation 
theory of hypersurfaces, there exists a deformation 
of $F$ with ordinary double points at $p_1,\dots, p_k$ and no other singularities. Indeed, if $G\in S_{d+1}$
is a general form vanishing at $p_1,\dots, p_k$ and non-vanishing at $p_{k+1},\dots, p_n$, then 
$F+tG \in \Sigma^{0}(p_1,\dots,p_k)$ will have ordinary double points
at $p_1,\dots, p_k$ and no other singularities  for $0<t\ll 1$. 

It remains to prove that $\Sigma_n$ is non-empty. Indeed, the following is an 
element of $\Sigma_n$:
\begin{equation*}
(d-1)(x_1+\cdots+x_n)^{d+1}-(d+1)(x_1+\cdots+x_{n})^{d-1}(x_1^2+\cdots+x_n^2)
+2(x_1^{d+1}+\cdots+x_n^{d+1}).
\end{equation*}
In fact, a generic linear combination of all degree $(d+1)$ monomials 
with the exception of $x_i^{d+1}$, for $i=1,\dots, n$, and $x_i^{d}x_j$, for $i, j=1,\dots, n$, $i<j$,
is a form with precisely $n$ ordinary double point singularities in general linear position. 
\end{proof}   

By Proposition \ref{P:A-Gr-defined}, we know that $\AA_{\Gr}$ is defined at all points of 
$Z_{1}\cup \cdots \cup Z_n$. In fact, we can explicitly compute $\AA_{\Gr}(U)$ for all
$U\in Z_n$, as well as the orbit closure of $\AA_{\Gr}(U)$ for all $U\in Z_{n-1}$. We need a preliminary fact.

\begin{prop}
\label{P:power}
Suppose $U\in \Grass(n, S_d)$ and $p\in \VV(I_U)\subset \PP V^{\vee}$. Let $L \in V^{\vee}$ be a non-zero linear form 
corresponding to $p$. Then $L^{n(d-1)-1} \in (I_U)_{n(d-1)-1}^\perp$. 
\end{prop}
\begin{proof}
Since $p\in \VV(I_U)$, all elements of $(I_U)_{n(d-1)-1}$ vanish at $p$, and it follows that $F\circ L^{n(d-1)-1}=0$ for all $F\in (I_U)_{n(d-1)-1}$ (cf.~Remark \ref{R:vanishing}). \end{proof}

\begin{corollary}
\label{C:Fermat} Suppose $U\in Z_k$ is such that 
\[
\VV(I_U)=\{p_1:=[1:0:\cdots:0], p_2:=[0:1:\cdots:0], \dots, p_k:=[0:\cdots:1:\cdots :0]\}.
\]
Then 
\[
\AA_{\Gr}(U) = \langle z_1^{n(d-1)-1}, \dots, z_k^{n(d-1)-1}, g_{k+1}(z_1,\dots,z_n), \dots, g_n(z_1,\dots,z_n) \rangle,
\]
for some $g_{k+1},\dots, g_n\in \D_{n(d-1)-1}$.
In particular, for $U\in Z_n$ one has
\[
\AA_{\Gr}(U) = \langle z_1^{n(d-1)-1}, \dots, z_n^{n(d-1)}\rangle=\nabla\bigl(\bigl[z_1^{n(d-1)}+\cdots+z_n^{n(d-1)}\bigr]\bigr).
\]
Moreover, for a generic $U\in Z_k$, we have $\AA_{\Gr}(U) \in \Grass(n, \D_{n(d-1)})_{\Res}$.
\end{corollary}
\begin{proof} 
Since the point $p_i=\VV(x_1,\dots, \widehat{x_i}, \dots, x_n) \in \PP V^{\vee}$ corresponds to the linear form $z_i \in V^{\vee}$,
Proposition \ref{P:power} implies that $z_i^{n(d-1)-1}\in \AA_{\Gr}(U)$ for every $i=1,\dots, k$. 

As $Z_n \subset \overline{Z}_{k}$ and $\AA_{\Gr}(U)\in \Grass(n, \D_{n(d-1)})_{\Res}$ for every 
$U\in Z_n$, it follows that $\AA_{\Gr}(U)$ is also spanned by a regular sequence for a generic $U\in Z_{k}$.
The claim follows.
\end{proof}


Consider the rational maps
\[
\begin{gathered}
\xymatrix{
\PP(S_{d+1})^{ss}\gitq \SL(n) \ar[rd]_{\overline{\nabla}} \ar@{-->}[rr]^{\bar{A}} & & \PP(\D_{n(d-1)})^{ss} \gitq \SL(n) \\
& \Grass(n, S_d)^{ss} \gitq  \SL(n) \ar@{-->}[ur]_{\bar{\AA}} &
}\label{D:rational-triangle-GIT}
\end{gathered}
\]
of projective GIT quotients.  

\begin{theorem}
\label{T:Fermat}
There is a dense open subset $Y_k$ of $Z_k$ such that  \[
\bar \AA \colon \Grass(n,S_d)^{ss} \gitq \SL(n) \dashrightarrow \PP(\D_{n(d-1)})^{ss}\gitq \SL(n)
\] is defined on $\pi_1(Y_k)$, $k=1,\dots,n$. Moreover, for $U\in Y_k$ the value $\bar \AA(\pi_1(U))$ is the image under $\pi_2$ of a polystable $k$-partial Fermat form. In particular, for every $U\in Z_{n}$ and for a generic $U\in Z_{n-1}$ 
\[
\bar \AA(\pi_1(U))=\pi_2\left(z_1^{n(d-1)}+\cdots+z_n^{n(d-1)}\right).
\]

\end{theorem}

\begin{proof} Recall that $Z_k$ is non-empty by Lemma \ref{L:non-empty}.
Suppose $U\in Z_k$ is generic, then by Corollary \ref{C:Fermat} in suitable coordinates we have
\[
\AA_{\Gr}(U) = \langle z_1^{n(d-1)-1}, \dots, z_k^{n(d-1)-1}, g_{k+1}(z_1,\dots,z_n), \dots, g_n(z_1,\dots,z_n) \rangle,
\]
and $\AA_{\Gr}(U)\notin \frakRes$.
It follows (as in the proof of \cite[Proposition 2.7]{fedorchuk-isaev}) that the closure of the $\SL(n)$-orbit of $\AA_{\Gr}(U)$
contains 
\begin{equation}\label{E:decomposable-A}
 \langle z_1^{n(d-1)-1}, \dots, z_k^{n(d-1)-1}, \bar g_{k+1}(z_{k+1},\dots,z_n), \dots, \bar g_n(z_{k+1},\dots,z_n) \rangle,
\end{equation}
where $\bar g_i:=g_i(0,\dots,0, z_{k+1},\dots, z_n)$ for $i=k+1,\dots, n$.
Then the claim follows for for $k=n-1$ and $k=n$ as in these cases we necessarily have $\bar g_n=z_n^{n(d-1)-1}$.

For $k$ arbitrary, since $\overline{\nabla}$ is a closed
immersion by Theorem \ref{T:gradient}, we conclude that $\bar \AA$ is defined at $\pi_1(U)$. Let 
$F\in \PP(\D_{n(d-1)})^{ss}$ be a polystable element with $\pi_2(F)=\bar\AA(\pi_1(U))$. Then we must have 
$\nabla(F) \in \overline{\SL(n)\cdot \AA_{\Gr}(U)}$, and so $\nabla(F)$ is linearly equivalent to an element of the form \eqref{E:decomposable-A}. It follows at once that 
\[
\bar \AA(\pi_1(U))=\pi_2\left(z_1^{n(d-1)}+\cdots+z_k^{n(d-1)-1}+G(z_{k+1},\dots,z_n)\right)
\]
is the image under $\pi_2$ of a polystable $k$-partial Fermat form.
\end{proof}

We will now establish Theorem \ref{T:barA1} as detailed in the next two corollaries.

\begin{corollary}\label{cor:n-1,n}
The rational map
\[
\bar A \colon (\PP S_{d+1})^{ss}\gitq \SL(n) \dashrightarrow \PP(\D_{n(d-1)})^{ss}\gitq \SL(n)
\]
is defined at a generic point of $\pi_0(\Sigma_{n-1})$ and at every point of $\pi_0(\Sigma_n)$. For a generic $f\in \Sigma_{n-1}$ and 
for every $f\in \Sigma_n$, we have 
\[
\bar A(\pi_0(f))=\pi_2(z_1^{n(d-1)}+\cdots+z_n^{n(d-1)}).
\]
\end{corollary}

\begin{corollary}
\label{C:contraction}
When $n=2$, the rational map $\bar A$ 
contracts the discriminant divisor to a point (corresponding to the orbit of the Fermat form in $\D_{2d-4}$) for all $d\ge 3$. When $n=3$, the rational map $\bar A$ contracts the discriminant divisor to a lower-dimensional subvariety if $d\ge 3$. More generally, for every $n\ge 4$ there exists $d_0$ such that for all $d\ge d_0$ the map $\bar A$ contracts the discriminant divisor to a lower-dimensional subvariety. 
\end{corollary}
\begin{proof} Notice that $\Sigma_1$ is dense in the discriminant divisor $\Delta$. Hence, for $n=2$ the statement follows from Corollary \ref{cor:n-1,n}.

When $n=3$, Theorem \ref{T:Fermat} implies that $\bar A(\pi_0(\Sigma_1))$ lies in the locus of a $1$-partial Fermat form in $\D_{3(d-1)}$.
The linear equivalence classes of $1$-partial ternary Fermat forms are in bijection with the linear equivalence classes of binary degree $3(d-1)$
forms. The dimension of this locus is $3d-6$, which for $d\ge 3$ is strictly less than the dimension $\binom{d+3}{2}-10$ of the discriminant divisor.

If $n\ge 4$, by Theorem \ref{T:Fermat} the set $\bar A(\pi_0(\Sigma_1))$ lies in the locus of a $1$-partial Fermat form in $\D_{n(d-1)}$. The linear equivalence classes of $1$-partial Fermat forms in $n$ variables are in bijection with the linear equivalence classes of degree $n(d-1)$
forms in $n-1$ variables. The dimension of this locus is ${n(d-1)+(n-2)}\choose{n-2}$, which for sufficiently large $d$ is strictly less than the dimension of the discriminant divisor ${{(d+1)+(n-1)}\choose{n-1}}-(n^2+1)$.
\end{proof}

We conclude the paper with an alternative proof of the main fact of \cite{alper-isaev-assoc} (see Proposition 4.3 therein).

\begin{corollary}[Generic smoothness of associated forms]
\label{C:generic-smoothness}
The closure of $\Image A$ in $\PP(\D_{n(d-1)})^{ss}$ contains the orbit 
\[
\SL(n)\cdot \left\{z_1^{n(d-1)}+\cdots+z_n^{n(d-1)}\right\}
\]
of the Fermat hypersurface.  Consequently, $A(f)$ is a smooth form for a generic smooth $f \in S_{d+1}$.
\end{corollary}
\begin{proof}
By Corollary \ref{cor:n-1,n}, we have
\[
\pi_2(z_1^{n(d-1)}+\cdots+z_n^{n(d-1)}) \in \Image(\bar A).
\]
Since the orbit of the Fermat hypersurface is closed in $\PP(\D_{n(d-1)})^{ss}$, it lies in the closure of $\Image A$.
\end{proof}

\bibliographystyle{plain}
\bibliography{associated_arxiv}{}

\begin{thebibliography}{10}

\bibitem{alper-isaev-assoc}
Jarod Alper and Alexander Isaev.
\newblock Associated forms in classical invariant theory and their applications
  to hypersurface singularities.
\newblock {\em Math. Ann.}, 360(3-4):799--823, 2014.

\bibitem{alper-isaev-assoc-binary}
Jarod Alper and Alexander Isaev.
\newblock {Associated forms and hypersurface singularities: The binary case}.
\newblock {\em J. reine angew. Math.}, 2016.
\newblock To appear, DOI: 10.1515/crelle-2016-0008.

\bibitem{eastwood-isaev1}
Michael Eastwood and Alexander Isaev.
\newblock Extracting invariants of isolated hypersurface singularities from
  their moduli algebras.
\newblock {\em Math. Ann.}, 356(1):73--98, 2013.

\bibitem{eastwood-isaev2}
Michael Eastwood and Alexander Isaev.
\newblock Invariants of {A}rtinian {G}orenstein algebras and isolated
  hypersurface singularities.
\newblock In {\em Developments and retrospectives in {L}ie theory}, volume~38
  of {\em Dev. Math.}, pages 159--173. Springer, Cham, 2014.

\bibitem{rota}
Richard Ehrenborg and Gian-Carlo Rota.
\newblock Apolarity and canonical forms for homogeneous polynomials.
\newblock {\em European J. Combin.}, 14(3):157--181, 1993.

\bibitem{eisenbud}
David Eisenbud.
\newblock {\em Commutative algebra with a view toward algebraic geometry},
  volume 150 of {\em Graduate Texts in Mathematics}.
\newblock Springer-Verlag, New York, 1995.

\bibitem{cayley-bacharach}
David Eisenbud, Mark Green, and Joe Harris.
\newblock Cayley-{B}acharach theorems and conjectures.
\newblock {\em Bull. Amer. Math. Soc. (N.S.)}, 33(3):295--324, 1996.

\bibitem{fedorchuk-direct}
Maksym {Fedorchuk}.
\newblock {Direct sum decomposability of polynomials and factorization of
  associated forms}, 2017.
\newblock \texttt{arXiv:1705.03452}.

\bibitem{fedorchuk-ss}
Maksym Fedorchuk.
\newblock G{IT} semistability of {H}ilbert points of {M}ilnor algebras.
\newblock {\em Math. Ann.}, 367(1-2):441--460, 2017.

\bibitem{fedorchuk-isaev}
Maksym {Fedorchuk} and Alexander {Isaev}.
\newblock {Stability of associated forms}.
\newblock {\em J. Algebraic Geom.}, to appear.
\newblock \texttt{arXiv:1703.00438}.

\bibitem{iarrobino-kanev}
Anthony Iarrobino and Vassil Kanev.
\newblock {\em Power sums, {G}orenstein algebras, and determinantal loci},
  volume 1721 of {\em Lecture Notes in Mathematics}.
\newblock Springer-Verlag, Berlin, 1999.

\bibitem{isaev-criterion}
Alexander Isaev.
\newblock A criterion for isomorphism of {A}rtinian {G}orenstein algebras.
\newblock {\em J. Commut. Algebra}, 8(1):89--111, 2016.

\bibitem{luna-slices-etales}
Domingo Luna.
\newblock Slices \'etal\'es.
\newblock {\em M\'emoires de la S. M. F.}, 33:81--105, 1973.

\bibitem{luna-adherences}
Domingo Luna.
\newblock Adh\'erences d'orbite et invariants.
\newblock {\em Invent. Math.}, 29(3):231--238, 1975.

\bibitem{GIT}
David Mumford, John Fogarty, and Frances Kirwan.
\newblock {\em Geometric invariant theory}, volume~34 of {\em Ergebnisse der
  Mathematik und ihrer Grenzgebiete (2) [Results in Mathematics and Related
  Areas (2)]}.
\newblock Springer-Verlag, Berlin, third edition, 1994.

\end{thebibliography}
\end{document}